\documentclass[11pt,twoside]{amsart}
\usepackage[utf8]{inputenc}
\usepackage[english]{babel}
\usepackage[T1]{fontenc}
\usepackage{enumitem}
\usepackage{amsmath,amsfonts,amssymb,mathtools}
\usepackage{xcolor}
\usepackage[hidelinks]{hyperref}
\usepackage{cleveref}
\usepackage[normalem]{ulem}

\definecolor{revision}{RGB}{20,65,175}

\newcommand{\bC}{\mathbb C}
\newcommand{\bR}{\mathbb R}
\newcommand{\bQ}{\mathbb Q}
\newcommand{\bZ}{\mathbb Z}
\newcommand{\bD}{\mathbb D}
\newcommand{\bT}{\mathbb T}
\newcommand{\cD}{\mathcal D}

\newcommand{\cM}{\mathcal M}

\newtheorem{theorem}{Theorem}[section]
\newtheorem{corollary}[theorem]{Corollary}
\newtheorem{lemma}[theorem]{Lemma}
\newtheorem{proposition}[theorem]{Proposition}
\theoremstyle{definition}
\newtheorem{definition}[theorem]{Definition}
\newtheorem{remark}[theorem]{Remark}

\title{Hyperk\"{a}hler SYZ conjecture}
\author[Engel]{Philip Engel}
\address{Department of Mathematics, Statistics, 
and Computer Science, University of Illinois Chicago (UIC),
851 S Morgan St, Chicago, IL 60607, USA}
\email{pengel@uic.edu}
\author[Mauri]{Mirko Mauri}
\address{Institut de Math\'ematiques 
de Jussieu-Paris Rive Gauche,
Universit\'e Paris Cit\'e,
Place Aur\'elie Nemours, 75013 Paris, France}
\email{mauri@imj-prg.fr}
\date{\today}

\thanks{PE was supported by NSF grant
DMS-2441240 and a Sloan Research Fellowship. MM was supported by Université Paris Cité and Sorbonne Université, CNRS, IMJ-PRG, F-75013 Paris, France.}

\begin{document}
\begin{abstract}
We prove the hyperk\"ahler SYZ conjecture: the sections of 
some power of a non-trivial nef isotropic line bundle define a Lagrangian fibration. 
\end{abstract}
\maketitle

\section{Introduction}

A compact irreducible hyperk\"{a}hler manifold $X$ is a simply-connected compact K\"{a}hler manifold, such that $H^0(X, \Omega^2_X)$ is generated
by an everywhere non-degenerate holomorphic two-form. The goal of this note is to prove the following theorem:

\begin{theorem}[Generalized abundance 
for compact hyperk\"{a}hler manifolds]\label{gen-abund}
Let $X$ be a compact irreducible hyperk\"{a}hler manifold, 
and let $L$ be a nef line bundle on $X$.
Then $L$ is semiample.
\end{theorem}

The second cohomology group $H^2(X,\bZ)$ 
admits a symmetric bilinear form $q$, the
Beauville--Bogomolov--Fujiki (BBF) form, with the following property: there exists a constant $c_X>0$ such that $\int_X \alpha^{2n} = c_X q(\alpha)^n$ for any $\alpha\in H^2(X, \bZ)$.
Theorem \ref{gen-abund} is known when
$q(L)>0$, i.e., $L$ is nef and big, by the base point
free theorem. 
Thus, the question is reduced
to the case 
$q(L)=0$. A fibration $f \colon X \to B$ with $0<\dim X < \dim B$ is Lagrangian, i.e., the general fiber is a holomorphic Lagrangian submanifold; see \Cref{def:holLagrangian} and \cite{Mat2001}.
Therefore, 
Theorem \ref{gen-abund}
is equivalent to the following assertion.

\begin{theorem}[Hyperk\"{a}hler SYZ conjecture] \label{thm:HKSYZ}
    Let $X$ be a compact irreducible hyperk\"{a}hler manifold, 
and let $L$ be a nontrivial nef line bundle on $X$ with $q(L)=0$. Then there exists a Lagrangian fibration $f \colon X \to B$ with $L^k = f^*A$ for some $k>0$ and an ample line bundle $A$ on $B$.
\end{theorem}

The conjecture is attributed to 
Tyurin and Bogomolov, and is stated explicitly by 
Hassett--Tschinkel, Huybrechts and Sawon; see \cite[Conj.~3.8 and Rem.~3.12]{HT2001}, \cite[\S 21.4]{Huy2003}, \cite[Conj.~4.1]{Sawon2003}, and also \cite{Verb2010}. Combining \Cref{thm:HKSYZ} with \cite[Cor.~1.3]{EFGMS25}, we obtain the following:

\begin{corollary}[Finiteness of hyperk\"{a}hler manifolds]\label{cor:finiteness}
The number of
deformation classes of compact irreducible hyperk\"{a}hler manifolds, of fixed dimension, 
with $b_2\geq 5$, is finite.
\end{corollary}

\begin{proof}
For $b_2\geq5$, Meyer's theorem gives 
a nonzero integral isotropic
class. A very general deformation keeping 
it of Hodge type $(1,1)$ has Picard
rank one. Thus, up to sign, it is nef, as 
there are no classes of negative
square in the Picard group. By \Cref{gen-abund}, 
this deformation admits a Lagrangian fibration. 
Then \cite[Thm.~B]{EFGMS25} gives finitely 
many deformation classes of
such manifolds in fixed dimension, 
proving the corollary.
\end{proof}

The strategy of the proof appears essentially in \cite[\S 3.5-\S 3.6]{HM22}; cf.~also \cite{GW2000}.
The key new ingredient is the metric SYZ conjecture, recently proved by Li \cite{li2026} and Blum--Liu \cite{blumli2026}. This result produces special Lagrangian tori in the large complex structure limit. A hyperk\"{a}hler rotation then transforms such a special Lagrangian torus into a holomorphic torus. Up to deformation, the holomorphic torus is a fiber of a Lagrangian fibration, which gives the required semiampleness. 

A special Lagrangian submanifold may never become holomorphic Lagrangian with respect
to any complex structure given by the hyperk\"{a}hler rotations; see \cite{Hattori2019}. To establish our result, we use the fact that the special Lagrangian torus above arises as the vanishing cycle of a maximal degeneration of compact hyperk\"{a}hler manifolds, or actually a cohomological incarnation of this geometry:~its Poincar\'{e} dual is a power of an isotropic vector; see \S\ref{sec:sLag}.

As a corollary, by hyperk\"{a}hler-rotating back holomorphic Lagrangian fibrations, we obtain a strengthening of the metric SYZ conjecture in the hyperk\"{a}hler context. The special Lagrangian fibration constructed in \cite{li2026} is a submersion on an open set of a Calabi--Yau variety near the large complex structure limit. In the hyperk\"{a}hler context, the submersion extends to a real analytic fibration with singular fibers on the whole variety.

\begin{corollary}[Strong SYZ conjecture]\label{thm:strongmetricSYZ} 
Let $\pi\colon (X, H) \to \bD^*$ be a polarized maximal degeneration of compact irreducible hyperk\"{a}hler manifolds, with $b_2 \geq 5$. For sufficiently small $t \in \bD^*$, there exists a 
real analytic
special Lagrangian fibration $f_t \colon X_t \to B_t$ with respect to the Calabi--Yau structure $
\left(X_t,\,\omega_{\mathrm{CY},t},\,\Omega_t\right)$ defined in \S 3. It is a hyperk\"{a}hler rotation of a holomorphic Lagrangian fibration. 
\end{corollary}

In particular, the discriminant of the fibration has real codimension two. Under the assumption that $B_t$ is a manifold, the Gromov--Hausdorff limit of $(X_t,\,\omega_{\mathrm{CY},t})$ is described in \cite[Thm.~5.2]{TZ2020}. In fact,
a recent pre-print \cite{oai-base} claims that
$B_t\simeq \mathbb{P}^n$, see also \cite{yoonjoo-base}.

\subsection{Acknowledgement} 
We are grateful to Radu Laza and Yuji Odaka for bringing the references \cite{GW2000} and \cite{OO2021} to our attention.

\subsection{AI disclosure}

A largely complete proof was human written 
and in the form of a rough draft. After the 
October 6, 2026 posting of manuscripts by OpenAI,
the authors used ChatGPT Astra to proofread 
the existing draft. A first version of this paper was posted on arXiv on October 8, 2026. The structure of our proof appears similar to that of \cite{oai-syz}.

\section{Special Lagrangian submanifolds}\label{sec:sLag}

A hyperk\"{a}hler manifold is a Riemannian manifold $(M,g)$ of real dimension $4n$ which is K\"{a}hler with respect to three complex structures $I$, $J$, $K$, satisfying the standard quaternion
relations $I^2=J^2=K^2=IJK=-\mathrm{Id}$.
If $(\omega_I, \,\omega_J, \,\omega_K)=(g(I \cdot , \cdot),\, g(J \cdot , \cdot), \, g(K \cdot , \cdot))$ denotes the corresponding K\"{a}hler forms, \[\sigma_{I}\coloneqq\omega_J+i\omega_K,
\qquad \sigma_{J}\coloneqq\omega_{K}+i\omega_I,\qquad \sigma_{K}\coloneqq \omega_{I}+i\omega_{J}\] are the holomorphic symplectic forms for $I$, $J$, $K$ respectively. We set 
\[\Omega_{I}\coloneqq \frac{\sigma^n_I}{n!},\] 
which is an $I$-holomorphic volume form.  In particular, $g$ is Ricci-flat.

\begin{definition}\label{def:slag}
Let $(X, I, \omega,\Omega)$
be a K\"ahler manifold of dimension $d$,
with Ricci-flat K\"ahler form $\omega$, and nowhere vanishing holomorphic
volume form $\Omega\in H^0(X, K_X)$.
An
oriented real submanifold $F\subset X$ of dimension $d$ is 
\emph{special Lagrangian} (of phase zero)
if $\omega|_F=0$,   $\Im(\Omega)|_F=0$, and ${\Re}(\Omega)|_F = C \cdot {\rm vol}_F$ for $C>0$, where ${\rm vol}_F$ is the volume form with
respect to $g=\omega(\cdot, I \cdot)$. 
\end{definition}

The third condition is automatic from
the first two, up to reversing the orientation of $F$; see \cite[\S I.8.1]{GHJ2003} or \cite[Cor.~III.1.11]{HB1982}. In particular, we can normalize $\Omega$ so that $\Omega\vert_F = \mathrm{vol}_F$.

\begin{definition}\label{def:holLagrangian}
Let $(M,g, I, J, K)$ be a hyperk\"ahler manifold of dimension $2n$.
An oriented submanifold $F\subset M$ is 
\emph{$J$-holomorphic Lagrangian} if it is a
$J$-complex submanifold of dimension $n$ and $\sigma_J|_F=0$.
\end{definition}

\begin{proposition}\label{lem:holspc}
Let $F\subset M$ be a compact special Lagrangian submanifold for
$(I,\omega_I,\Omega_I)$. Then the following are equivalent:
\begin{enumerate}
\item $F$ is $J$-holomorphic Lagrangian;
\item $[\omega_K|_F]=0$ in $H^2(F, \bR)$;
\item $\int_F (\omega_{J}+i \omega_{K})^n =  \int_F \omega^n_J$.
\end{enumerate}
\end{proposition}

Before proving \Cref{lem:holspc}, we recall the classical Wirtinger's inequality.
\begin{lemma}[Wirtinger's inequality]  Let $(M, g, \omega, I)$ be a complex K\"{a}hler manifold and $F \subset M$ be an orientable real submanifold of dimension $2n$. Let $\mathrm{vol}_{F}$ denote the Riemannian volume form of the induced metric $g|_{F}$. Then 
  $$\bigg| \frac{\omega^{n}}{n!} \bigg|\big|_{F}\le
  \mathrm{vol}_{F},$$
  and equality holds if and only if $F$ is a complex submanifold.
\end{lemma}

\begin{proof}
See \cite[II.3, Lem.~3.5, Ex.~I]{HB1982}, \cite[1.8.2]{Federer1969} 
or \cite[Thm~2.2]{zhu2023}.
\end{proof}

\begin{proof}[Proof of \Cref{lem:holspc}]
$(1)\Rightarrow (2)$:   $\omega_{K}|_{F} = \Im(\sigma_I|_F)=0$. $(2)\Rightarrow (3)$:  trivial. $(3)\Rightarrow (1)$: Up to normalizing $\Omega_{I}$, as 
$F$ is special Lagrangian, we have
$\Omega_{I}\vert_F = \mathrm{vol}_{F}$.  
Integrating along $F$, we obtain 
\[
\mathrm{Vol}(F)=
\int_F
\Omega_I  = \frac{1}{n!}\int_F
(\omega_{J} + i \omega_K)^n  = \frac{1}{n!} \int_{F} \omega_J^n  \leq 
\int_{F} \bigg|\frac{\omega_J^n}{n!} \bigg|
\leq \mathrm{Vol}(F),\]
by (3) and Wirtinger's inequality. Equality throughout forces the equality case in Wirtinger's inequality, i.e., $F$ is an $J$-complex submanifold. 

 In particular, $\sigma_{J}|_{F} =(\omega_{K}+i\omega_I)|_{F}= \omega_{K}|_{F}$ is a complex (2,0)-form. The identity $\sigma_{J}(Jv, w)=i\sigma_{J}(v,w)$, with $v,w \in TF$, gives $\omega_{K}(Jv, w)=i\omega_{K}(v,w)$, but the left side is real and the right side is imaginary, so $\sigma_{J}|_{F} = \omega_{K}|_{F}=0$. We conclude that $F$ is $J$-holomorphic Lagrangian. 
\end{proof}
 Let $M$ be a compact irreducible hyperk\"{a}hler manifold of complex dimension $2n$. 
 Let $q$ be the Beauville--Bogomolov--Fujiki form and $c_M>0$ the Fujiki constant such that
 \[
 \int_M \alpha^{2n} = c_M q(\alpha)^n \quad \text{ with }\alpha\in H^2(M, \bC).
 \]

\begin{lemma} \label{lem:technico}
Let $(M, g, I, J, K)$ be a compact irreducible hyperk\"{a}hler manifold of complex dimension $2n$. Let $F$ be a compact special Lagrangian submanifold for
$(I,\omega_I,\Omega_I)$ such that
\[
\mathrm{PD}[F]=c\lambda^n \quad \text{ with }\lambda \in H^2(M, \bQ), \quad q(\lambda)=0, \quad c\in \bQ.
\]
Let $c'\coloneqq c\,c_M\,2^n{2n\choose n}^{-1}$. Then, we have \[\int_F \alpha^n = c' q(\lambda, \alpha)^n\] for any $\alpha\in H^2(M, \bC)$. Furthermore,
$c \neq 0$, $q(\lambda, \sigma_{I}) \neq 0$, $q(\lambda, \sigma_{I})^n\in \bR$, and  $q(\lambda, \omega_I)=0$.
\end{lemma}

\begin{proof}
Apply Fujiki's relation $\int_M (s \lambda + \alpha)^{2n} = c_M q(s \lambda + \alpha)^n=
c_M(2s q(\lambda, \alpha) +q(\alpha))^n$ and compare the coefficients of $s^n$:
\[
{2n \choose n} \int_F \alpha^n = c\,{2n \choose n} \int_M \lambda^n \alpha^n = c\,c_M\,2^n q(\lambda, \alpha)^n.
\]
Thus, $c \neq 0$, $q(\lambda, \sigma_{I}) \neq 0$, and $q(\lambda,\sigma_I)^n\in \bR$, since the special Lagrangian property entails \[\mathrm{PD}[F] \cdot \Omega_{I} = \int_F \Omega_{I} = \int_F \frac{\sigma^n_{I}}{n!} =c' \frac{q(\lambda, \sigma_I)^n}{n!}=\mathrm{Vol}(F)> 0.\]
Finally,
$0=\int_F \omega_I^n = c' q(\lambda, \omega_I)^n$ 
because $F$ is $\omega_I$-Lagrangian.
\end{proof}

\begin{proposition}\label{prop:PDholLag}
Let $(M, g, I, J, K)$ be a compact irreducible hyperk\"{a}hler manifold of complex dimension $2n$. Let $F$ be a compact special Lagrangian submanifold for
$(I,\omega_I,\Omega_I)$ such that
\[
\mathrm{PD}[F]=c\lambda^n \quad \text{ with }\lambda \in H^2(M, \bQ), \quad q(\lambda)=0, \quad c\in \bQ.
\]
Then, up to multiplying $\sigma_I$ by an $2n$-th root 
of unity, i.e., replacing $J\mapsto \cos(\tfrac{\pi r}{n})J - \sin(\tfrac{\pi r}{n})K$, $K\mapsto 
\sin(\tfrac{\pi r}{n})J + \cos(\tfrac{\pi r}{n})K$, and possibly
reversing the orientation
of $F$, we have that $F$ is a $J$-holomorphic
Lagrangian. Moreover, 
$\lambda \in H^{1,1}(M, J)$  and 
$q(\lambda,\omega_{J})>0$.
\end{proposition}

\begin{proof}
By \Cref{lem:technico}, $q(\lambda,\sigma_I)^n$
is real. There
is thus a unique $2n$-th root of unity $\zeta_{2n}^r$ so that,
after replacing $\sigma_I$ with $ \zeta_{2n}^r \sigma_I$, we have $q(\lambda,\sigma_I)=q(\lambda,\omega_J)>0$ positive
and real. If $r$ is odd, the volume form changes
from $\Omega_I$ to $-\Omega_I$ and
$F$ with its opposite
orientation is special Lagrangian for
the new volume form. Otherwise, if $r$ is even,
$\Omega_I$ is unchanged,
and we keep the same orientation for $F$.
Having made these substitutions,
we apply \Cref{lem:technico} to show that \begin{equation}\label{eq:identity2}\int_F (\omega_J+i\omega_K)^n =  c'q(\lambda, \omega_J+i\omega_K)^n = 
c'q(\lambda, \omega_J)^n = 
\int_F \omega_J^n.\end{equation}
By \Cref{lem:holspc}, 
$F$ is 
$J$-holomorphic Lagrangian. 
By the last part of \Cref{lem:technico},
$q(\lambda, \omega_I)=0$. Since
$q(\lambda,\omega_K)=0$, we also
have $q(\lambda,\sigma_J)=0$, and 
so $\lambda\in H^{1,1}(M,J)$.
\end{proof}

\section{Classical SYZ conjecture}\label{metric-syz}

Let $\pi\colon (X, H) \to \bD^*$ be a degeneration of Calabi--Yau manifolds over the punctured disk with coordinate $t$, polarized by a relatively ample line bundle $H$. Assume that $\pi$ is an analytic restriction
of a projective family over a punctured algebraic curve.
By semistable reduction, after a finite base change, we may assume that $X$ also admits a semistable filling
over the disk $\bD$.
By \cite[Thm.~1.1]{Fujino2011}, a minimal model program for the semistable filling ends with a minimal dlt model $(\bar{X}, X_{0}) \to \bD$, i.e., $K_{\bar{X}}\sim 0$ and $(\bar{X}, X_{0})$ is dlt.

The degeneration $\pi$ is maximal if the dual complex $\cD(X_0)$ has maximal dimension $\dim X_t\eqcolon d$. In particular, $X_0$ then has a stratum of dimension $0$. In a neighbourhood of this stratum, $(\bar{X}, X_{0})$ is analytically locally isomorphic to $(\bC^{d+1}, \,x_{0}\ldots x_{d}=0)$, and $X_{t}$ is cut by the equation $x_{0}\ldots x_{d}=t$. A profound torus is a real torus $\bT^{d} \subset X_{t}$, isotopic to 
\[
\big\{(r_0e^{2\pi i\theta_0},\ldots,r_de^{2\pi i\theta_d}): 
\theta_j\in\bR/\bZ, \,\, \textstyle \prod^d_{j=0}r_je^{2\pi i\theta_j}=t
\big\}
\]
 for fixed small real numbers $r_j>0$ satisfying $\prod_{j=0}^d r_j = |t|$.
\begin{remark}\label{rmk:profound} The ambient isotopy type of $\bT^{d}$ does not depend either on the choice of the
coordinates $(x_0, \ldots, x_d)$ or on the choice of the minimal stratum of $X_0$; see \cite[Rmk 3.7]{HM22}. Profound tori can be defined also in a neighborhood of a stratum of dimension $0$ of a semistable model, but in general its ambient isotopy class depends on the stratum, and it is independent if it is a minimal model. By \cite[Lem.~16]{dFKX17}, a neighbourhood of a stratum of dimension $0$ in $X_0$ is isomorphic to a neighbourhood of a stratum of dimension $0$ in the semistable model, so every profound torus in the dlt model is isotopic to a profound torus in the semistable model.
\end{remark}

Assume that there is a nowhere vanishing holomorphic volume form $\Omega \in H^0(X, K_X)$  inducing the fiberwise holomorphic volume forms $ \Omega_t \coloneqq \iota_{d/dt}\Omega\vert_{X_t}$.  These
forms may be uniquely normalized so that 
$\int_{\bT^d} \Omega_t =1$
for all $t\in \bD^*$.
We consider the Calabi--Yau metric $\omega_{\mathrm{CY},t}$ on $X_t$ in the ample class
$\frac{1}{|\!\log |t||}\,c_1(H)$,
defined by the complex Monge--Ampère equation
\[
\omega_{\mathrm{CY},t}^d
=
\frac{(H^d)}{|\!\log |t||^{d}}\mu_t, \quad \text{ with }\mu_t
=
\frac{\Omega_t\wedge\overline{\Omega}_t}
{\int_{X_t}\Omega_t\wedge\overline{\Omega}_t}.
\]

\begin{theorem}\label{thm:metricSYZ} \cite{li2026, blumli2026}
Let $\pi\colon (X, H) \to \bD^*$ be a polarized maximal degeneration of Calabi--Yau manifolds.  Given any
$0<\delta\ll 1$,
for sufficiently small \(t\), depending on \(\delta\), there exists a special Lagrangian  torus fibration with respect to the Calabi--Yau structure $
\left(X_t,\,\omega_{\mathrm{CY},t},\,\Omega_t\right),$ on an open subset  of \(X_t\) whose normalised Calabi--Yau measure is at least
$1-\delta$. The fiber of the fibration is isotopic to a profound torus, and the base of the fibration is an open subset of $\cD(X_{0})$. 
\end{theorem}

\begin{proof} The first part of the statement follows from \cite[Conjecture 1.1, p.~4]{li2026}. 
The normalization $\int_{\bT^d}\Omega_t=1$
ensures that the phase of the special Lagrangians is zero,
as in Definition \ref{def:slag}.
Near a minimal stratum, the fibration is a small
perturbation of the torus fibration
\[(x_1, \ldots, x_d) \mapsto \frac{1}{\log|t|}
(\log|x_1|, \ldots, \log|x_d|),\]
so the fibers are isotopic to profound tori.
\end{proof}

\begin{lemma}\label{lem:weight}
Let $\pi\colon X\to\bD^*$ be a maximal 
degeneration of compact
irreducible hyperk\"{a}hler manifolds of complex dimension $2n$, with
limit mixed Hodge
structure $(H^*_{\rm lim}, W_{\bullet}H^*_{\rm lim})$ on $H^*(X_t,\bQ)$.
There exists a nonzero class $\lambda \in H^2(X_t,\bQ)$, isotropic for $q$, such that 
$$
W_0H^2_{\mathrm{lim}}=\bQ \lambda,\qquad
W_0H^{2n}_{\mathrm{lim}}=\bQ\lambda^n.
$$ 
The Poincar\'e dual of every profound torus lies in  $W_0H^{2n}_{\mathrm{lim}}= \bQ\lambda^n$. If it is isotopic to a special Lagrangian torus,
this class is nonzero.
\end{lemma}

\begin{proof}
After a finite base change, the monodromy operator $T_k$ on $H^k(X_t, \bQ)$ is unipotent, and $T_2$ preserves $q$, i.e., $q(T_2 \alpha, T_2 \alpha')=q(\alpha, \alpha')$ with $\alpha, \alpha' \in H^2(X_t, \bQ)$. The logarithmic monodromy operator $N_k = \log T_{k}$ on  $H^k(X_t, \bQ)$ is nilpotent of nilpotency index  $k+1$ for even $k \leq 2n$, i.e., $N^{k}_{k} \neq 0$ but $N^{k+1}_{k} = 0$,  and $q(N_2 \alpha, \alpha')=-q(\alpha, N_2\alpha')$; see \cite[Prop.~7.14]{KLSV2018} and \cite[Lem.~2.1]{HM23}. Hence, $W_0H^2_{\mathrm{lim}}=\mathrm{Im}(N^2_2)$ is isotropic since 
$q(N^2_2 \alpha, N^2_2 \alpha')=-q(N_2\alpha, N^3_2\alpha')=0$. 
Moreover, 
$$W_0H^{2k}_{\rm lim}\simeq W_0 {\rm Sym}^k H^2_{\rm lim}\simeq {\rm Sym}^kW_0H^2_{\rm lim} \simeq \bQ \lambda^k$$ for all $k\leq n$.
Indeed, we have ${\rm Sym}^k H^2_{\rm lim}\subset H^{2k}_{\rm lim}$ for $k\leq n$ by \cite[Thm.~1.5]{Ver96}. The Hodge filtration on
its orthogonal complement has level $<2k$, and  therefore the weights of this complement
are supported between $2$ and $4k-2$; see also \cite[Proof of Prop.~7.14]{KLSV2018}.

By Clemens' description of maximal vanishing cycles of semistable fillings in
\cite{Clemens1969} and \cite[Thms.~3.5,~3.8]{Harder19}, the fundamental class of a profound torus $\bT^{2n}$ lies in $W_{-4n}H_{2n}(X_t, \bQ)$, so its Poincar\'{e} dual lies in $W_0H^{2n}_{\mathrm{lim}}$; cf.~\Cref{rmk:profound} for the comparison of semistable and dlt fillings.   
Finally, a special Lagrangian torus $F$ has
$\int_{F}\Omega_t >0$, so its class is nonzero.
\end{proof}

\section{Hyperk\"{a}hler SYZ conjecture}
\begin{proof}[Proof of \Cref{thm:HKSYZ}]
Let $L$ be a nef line bundle on a compact irreducible hyperk\"{a}hler manifold $M$.
If $q(L)>0$, then $M$ is projective
and $L$ is big, and the semiampleness follows from the base point free theorem. So assume that $L$ is non-trivial with $q(L)=0$.

 A {\it marking} of $M$ is an isometry from 
$H^2(M,\bZ)$ to a fixed reference lattice. 
Let $\mathcal{M}$ 
denote a connected component 
of the moduli space of marked compact irreducible
hyperk\"ahler manifolds containing $M$
with a given marking.
Let $\mathcal{M}_\lambda\subset \mathcal{M}$ 
be the sub-moduli space
on which a reference isotropic class $\lambda$ is the first 
Chern class $\lambda = c_1(L)$ of a nef line bundle $L$. It is connected by \cite[Prop.~2.7]{SV2025}. By \cite[Thm.~3.8]{SV2025}, it suffices to prove that $L$
is semiample on one member of $\mathcal{M}_\lambda$.

We first construct a projective maximal degeneration $\pi\colon (X, H) \to \bD^*$ 
of compact irreducible hyperk\"{a}hler
manifolds, admitting a classifying
map $\mathbb{H}\to \mathcal{M}$ from the universal cover $\mathbb{H}$ of $\bD^*$, and such that
$W_0 H^2_{\mathrm{lim}}=\bQ\lambda$, with respect
to the marking arising from
this map. We follow \cite[\S 4]{Sold2019}. Choose an
integral vector $\mu\in \lambda^\perp$
with $q(\mu)>0$ and
define a nilpotent operator $N\in \frak{so}(H^2(M,\bZ))$
by the formula
$N x \coloneqq  q(x, \mu)\lambda - q(x, \lambda)\mu.$ 
Then $N^3=0$, $N^2\neq 0$. 

Fix some $v\in \{\lambda,\mu\}^\perp$ integral,
 with $q(v)>0$, and choose a line $\bC z \in \mathbb{P}(v^{\perp_q}\otimes \bC)$, 
 with $q(z, \lambda)=1$, $q(z, z)=0$. The period
mapping $\sigma_t \coloneqq {\rm exp}(\frac{\log t}{2\pi i} N)\cdot z$ is well-defined on
$\mathbb{H}$ for small $t$ and
induces a map $t\mapsto \bC \sigma_t$ 
to the period space
for $v$-quasipolarized hyperk\"ahler
manifolds. This map descends to
$\bD^*$ for $\mu$ sufficiently divisible, i.e., as long as ${\rm exp}(N)$ is a monodromy operator
of the unmarked moduli space; cf.~\cite[Thm.~2.6]{Verb20} and \cite[Thm.~8.2]{BL2022}.
Hence, by global Torelli \cite{Verb13, huyglob}, we get
a degeneration $X \to \bD^*$ of hyperk\"ahler manifolds.
Furthermore, there is a
choice of marking of $X_t$ 
for which the multivalued
period map on $\bD^*$
is $t\mapsto \bC \sigma_t$ and
the classifying map lands in 
$\mathbb{H}\to \cM$, identifying
$W_0H^2_{\rm lim}$ with
the chosen isotropic ray $\bQ \lambda$.

By algebraic approximation, we may ensure
that $X\to \bD^*$ is an analytic restriction
of an algebraic 
family with the same 
logarithmic monodromy $N$
and very general Picard rank $1$.
Possibly after shrinking $\bD^*$ and taking a finite cover, the conditions $q(v)>0$, $q(v, \sigma_t)=0$,
and the Picard rank $1$ condition ensure that
$v$ defines a relatively ample line bundle.
So $X\to \bD^*$ is projective.
After a further finite base change, 
it admits a semistable filling, and 
the hypotheses of \S\ref{metric-syz} are satisfied. The degeneration is maximal by \cite[Prop.~7.14.(ii)]{KLSV2018}.

By \Cref{thm:metricSYZ}
and \Cref{lem:weight}, 
for sufficiently small $t$, 
$(X_t, I, \omega_I \coloneqq
\omega_{{\rm CY},t}, \Omega_t)$ 
contains a special Lagrangian torus
$F$ with $\mathrm{PD}[F]=c\lambda^n$, $c \neq 0$,  isotopic to
a profound torus. 
The Ricci-flat metric on $X_t$ is hyperk\"{a}hler for
the compatible triple $(I,J,K)$ defined
by taking $\sigma_I = C \sigma_t$ for 
the unique positive real number $C>0$
ensuring $\omega_I$, $\omega_J$, $\omega_K$ 
have equal volume.
Its volume form
$\Omega_I=\sigma_I^n/n!$ is a positive real 
multiple of $\Omega_t$,
so $F$ remains special Lagrangian for $\Omega_I$. 
Furthermore, by construction, $q(\lambda, \sigma_I)>0$
is real and positive for all $t$.

By \Cref{prop:PDholLag},  
$F$ is a $J$-holomorphic Lagrangian torus
on $(X_t,J)\in \mathcal{M}$,
and furthermore, $\lambda \in H^{1,1}(X_t, J)$. 
So there is a
line bundle with $c_1(L)=\lambda$.
Up to a small deformation of 
$(X_t,J,L,F)$, we can
suppose that $(X_t,J)$ is non-projective with
$\mathrm{Pic}(X_t, J)_{\bQ} = \bQ L$; 
see \cite[Prop.~2.1]{Matsushita}. 
Hence $L$ is nef\footnote{A posteriori, in view of the proof of \Cref{thm:strongmetricSYZ}, if $b_2\geq 5$ and $t$ sufficiently small, $L$ is on the nose nef on $(X_t, J)$, and the small deformation is not necessary.} and 
$(X_t,J)\in \mathcal{M}_\lambda$.
By \cite[Thm 4.1]{GLR2013}, there exists a Lagrangian fibration $f \colon (X_t, J) \to B$ with fiber $F$ and normal projective base $B$. The pullback of an ample line bundle of $B$ must be a multiple of $L$. Thus $L$ is semiample on $(X_t, J)$. As explained above, by \cite[Thm.~3.8]{SV2025},
$L$ remains semiample in $\mathcal{M}_\lambda$. This proves the SYZ conjecture.
\end{proof}

\section{Strong SYZ conjecture}

\begin{proof}[Proof of \Cref{thm:strongmetricSYZ}] Let $\pi\colon (X, H) \to \mathbb{D}^*$ be a polarized maximal degeneration of compact irreducible hyperk\"{a}hler manifolds with logarithmic monodromy $N$. Let $E \colon \mathbb{H} \to \mathbb{D}^*$ be the universal cover given by $\tau = x+iy \mapsto t=\exp(2\pi i \tau)$. The pullback $E^*R^2\pi_*\mathbb{Q}$ is a trivial local system, and a trivialization induces families of markings $H^2(X_\tau, \bZ) \simeq (\Lambda,q)$. Let $v \coloneqq c_1(H_{\tau})\in \Lambda$ be the relative polarization with $q(v)>0$, and $\lambda \in \Lambda$ be the isotropic generator of $W_0H^2_{\mathrm{lim}}$. We write $N(x) =  q(x, \mu)\lambda - q(x, \lambda)\mu$ for some $\mu \in \lambda^{\perp}$ with $q(\mu)>0$.

By Schmid's nilpotent
orbit theorem \cite[Thm.~4.9]{Schmid}, 
the period map of $\pi$ is $$ \tau\mapsto \sigma_{\tau} \coloneqq {\rm exp}(\tau  N)\cdot z(t) = -\tfrac{1}{2}\tau^2 q(\mu)\lambda + \tau q(z(t), \mu)\lambda-\tau\mu+z(t)$$ for some holomorphic map $z = a+i b \colon \bD\to \mathbb{P}\{x\in v^\perp\otimes \bC \,\big{|}\, q(x)=0\}$. The period $\sigma_\tau \in v^\perp\otimes \bC$ 
is defined up to a complex scalar. 
We can normalize $z(t)$ so that $q(\lambda, \sigma_\tau)=q(\lambda,z(t))=1$.
Note that $q(\Re(\sigma_\tau)) = q(\Im(\sigma_\tau))$
 and set $c_\tau\coloneqq q(\Re(\sigma_\tau))^{-1/2}$.

We may hyperk\"ahler rotate \begin{align*}&\big(X_\tau, &&I_\tau, &&\omega_{I_\tau} \coloneqq q(v)^{-1/2}v ,&&\sigma_{I_\tau}\coloneqq c_\tau\Re(\sigma_\tau) +ic_\tau\Im(\sigma_\tau)), &\textrm{to} \\
&\big(X_\tau, &&J_\tau, && \omega_{J_\tau} \coloneqq c_\tau\Re(\sigma_{\tau}),&& \sigma_{J_\tau}\coloneqq c_\tau \Im(\sigma_{\tau})  +i q(v)^{-1/2}v).\end{align*}
By \Cref{prop:PDholLag}, $\lambda \in H^{1,1}(X_{\tau}, J_{\tau})$. Our goal is to show that $\lambda$ is nef on $(X_{\tau}, J_{\tau})$. If so, then by \Cref{thm:HKSYZ}, $(X_{\tau}, J_{\tau})$ carries a holomorphic Lagrangian fibration $f \colon (X_{\tau}, J_{\tau}) \to B$, which can be hyperk\"ahler
rotated back to $(X_{\tau}, I_{\tau})$ 
to produce a special Lagrangian fibration
filling out all of $X_{\tau}$  and extending
the special Lagrangian fibrations appearing in Theorem \ref{thm:metricSYZ}.

Nef line bundles on hyperk\"ahler manifolds are characterized by having non-negative intersection with MBM classes; see \cite[Thm 6.2]{AV21}. By \cite[\S 3.2]{AV20}, since $b_2 \geq 5$, a MBM class $\beta$ on $(X_{\tau}, J_{\tau})$  enjoys the following properties:
\begin{enumerate}
    \item \label{cond1} $\beta$ lies in $H^{1,1}(X_{\tau}, J_{\tau}) \cap H^2(X_{\tau}, \bZ)$, i.e., $q(\beta, \Im(\sigma_{\tau}))=q(\beta, v)=0$;
    \item \label{cond2} a multiple of $q(\beta, \cdot)$ is the class of an effective curve. In particular, $q(\beta, \Re(\sigma_{\tau}))>0$;
    \item \label{cond3} $-B < q(\beta) < 0$ for a fixed constant $B$, depending only on the deformation class of $X$.
\end{enumerate}
We show that, for sufficiently small $t$, the inequality $q(\beta, \lambda)<0$ forces $q(\beta, \Re(\sigma_{\tau}))<0$, which is a contradiction, so $\lambda$ is nef. To this end, take real and imaginary part of $\sigma_\tau$:
\begin{align*} \Re (\sigma_{\tau}) &= -\tfrac{1}{2} (x^2-y^2)q( \mu)\lambda+q(xu-yv, \mu) \lambda -x\mu+u. \\
\Im(\sigma_{\tau})&=(q(xv+yu, \mu)-xyq(\mu))\lambda -y\mu +v.
\end{align*}
We now restrict $\tau$ to the Siegel domain $\mathfrak{S} \coloneqq \{ \tau=x+iy \in \mathbb{H} \colon  0<x<1, \,y>1 \}$. 
Choose a real splitting of the filtration $\bR \lambda \subset \lambda^{\perp} \subset \Lambda_{\bR}$ as follows: $\Lambda_{\bR} \simeq \bR \lambda \oplus \Lambda' \oplus \bR \nu$, where $\nu$ is an isotropic class in $\Lambda_{\bR}$ with $q(\nu,\lambda)=1$, and $\Lambda' =\langle \nu, \lambda \rangle^{\perp}$. In this base, we write
\begin{align*} 
\omega_{J_\tau}  =c_\tau \Re(\sigma_{\tau}) = N_{\tau} \lambda + w_{\tau} + \epsilon_\tau \nu,\qquad
\Re (\sigma_{J_{\tau}}) = c_\tau \Im(\sigma_{\tau}) =  c_{\tau} \lambda + v_{\tau}.
\end{align*}
For $y \to \infty$, we have $q(\Re(\sigma_\tau))=q(\Im(\sigma_\tau))\sim \frac{2}{q(\mu)} y$, so in the limit we have also:
\begin{enumerate}[label=(\alph*)]
    \item \label{eq:bound1}  $y \cdot w_{\tau} \sim \sqrt{\frac{2}{q(\mu)}}(-x\mu +u-q(u,\nu)\nu)$ and $v_{\tau} \sim \sqrt{\frac{2}{q(\mu)}}\mu$ lie in a bounded set of $\Lambda'_{\bR}$;
    \item \label{eq:bound2} $\epsilon_{\tau} \sim \sqrt{\frac{2}{q(\mu)}}q(u,\nu) \cdot \frac{1}{y} \to 0$;
    \item \label{eq:bound3} $c_{\tau} \sim \sqrt{\frac{2}{q(\mu)}}(q(u, \mu)-xq(\mu))$ is bounded.
\end{enumerate}
The argument of \cite[Claim 4.18]{OO2021} then applies: There exists a constant $B'$, only depending on the bounds \ref{eq:bound1} and \ref{eq:bound2}, such that if \eqref{cond1} and \eqref{cond2} hold and further $q(\beta, \lambda)=-m <0$, i.e., $\lambda$ is not nef, then $q(\beta, \Re(\sigma_{\tau}))< (-\frac{1}{3\epsilon_{\tau}}+B' \epsilon_{\tau}) \cdot m$, thus violating \eqref{cond2} for $\epsilon_{\tau}$ sufficiently small. Note that in {\it loc.~cit.}~$q(\beta)=-2$ and $v_{\tau}$ is constant, 
but boundedness suffices; cf.~\cite[\S 8.4, Step 3]{OO2021}.
\end{proof}

\bibliography{biblio}

\begin{thebibliography}{10}

\bibitem{AV20}
E.~Amerik and M.~Verbitsky, \emph{Collections of orbits of hyperplane type in homogeneous spaces, homogeneous dynamics, and hyperk\"{a}hler geometry}, Int. Math. Res. Not. IMRN  (2020), no.~1,  25--38.

\bibitem{AV21}
---{}---{}---, \emph{Rational curves and {MBM} classes on hyperk\"{a}hler manifolds: a survey}, in Rationality of varieties, Vol. 342 of \emph{Progr. Math.}, 75--96 (2021).

\bibitem{BL2022}
B.~Bakker and C.~Lehn, \emph{The global moduli theory of symplectic varieties}, J. Reine Angew. Math. \textbf{790} (2022) 223--265.

\bibitem{blumli2026}
H.~Blum and Y.~Liu, \emph{Valuative independence for Calabi--Yau varieties}, arXiv:2604.27890 (2026).

\bibitem{Clemens1969}
C.~H. Clemens, Jr., \emph{Picard-{L}efschetz theorem for families of nonsingular algebraic varieties acquiring ordinary singularities}, Trans. Amer. Math. Soc. \textbf{136} (1969) 93--108.

\bibitem{dFKX17}
T.~de~Fernex, J.~Koll\'{a}r, and C.~Xu, \emph{The dual complex of singularities}, in Higher dimensional algebraic geometry---in honour of {P}rofessor {Y}ujiro {K}awamata's sixtieth birthday, Vol.~74 of \emph{Adv. Stud. Pure Math.}, 103--129, Math. Soc. Japan, Tokyo (2017).

\bibitem{EFGMS25}
P.~Engel, S.~Filipazzi, F.~Greer, M.~Mauri, and R.~Svaldi, \emph{Boundedness of some fibered K-trivial varieties}, arXiv:2507.00973 (2025).

\bibitem{Federer1969}
H.~Federer, Geometric measure theory, Die Grundlehren der mathematischen Wissenschaften, Band 153, Springer-Verlag New York, Inc., New York (1969).

\bibitem{Fujino2011}
O.~Fujino, \emph{Semi-stable minimal model program for varieties with trivial canonical divisor}, Proc. Japan Acad. Ser. A Math. Sci. \textbf{87} (2011), no.~3,  25--30.

\bibitem{GLR2013}
D.~Greb, C.~Lehn, and S.~Rollenske, \emph{Lagrangian fibrations on hyperk\"{a}hler manifolds---on a question of {B}eauville}, Ann. Sci. \'{E}c. Norm. Sup\'{e}r. (4) \textbf{46} (2013), no.~3,  375--403 (2013).

\bibitem{GHJ2003}
M.~Gross, D.~Huybrechts, and D.~Joyce, Calabi-{Y}au manifolds and related geometries, Universitext, Springer-Verlag, Berlin (2003). Lectures from the Summer School held in Nordfjordeid, June 2001.

\bibitem{GW2000}
M.~Gross and P.~M.~H. Wilson, \emph{Large complex structure limits of {$K3$} surfaces}, J. Differential Geom. \textbf{55} (2000), no.~3,  475--546.

\bibitem{Harder19}
A.~Harder, \emph{Torus fibers and the weight filtration}, arXiv:1908.05110 (2019).

\bibitem{HB1982}
R.~Harvey and H.~B. Lawson, Jr., \emph{Calibrated geometries}, Acta Math. \textbf{148} (1982) 47--157.

\bibitem{HT2001}
B.~Hassett and Y.~Tschinkel, \emph{Rational curves on holomorphic symplectic fourfolds}, Geom. Funct. Anal. \textbf{11} (2001), no.~6,  1201--1228.

\bibitem{Hattori2019}
K.~Hattori, \emph{New examples of compact special {L}agrangian submanifolds embedded in hyper-{K}\"{a}hler manifolds}, J. Symplectic Geom. \textbf{17} (2019), no.~2,  301--335.

\bibitem{Huy2003}
D.~Huybrechts, \emph{Compact hyperk\"{a}hler manifolds}, in Calabi-{Y}au manifolds and related geometries ({N}ordfjordeid, 2001), Universitext, 161--225, Springer, Berlin (2003), ISBN 3-540-44059-3.

\bibitem{huyglob}
---{}---{}---, \emph{A global {Torelli} theorem for hyperk\"ahler manifolds [after {M.} {Verbitsky]}}, in S\'eminaire Bourbaki Volume 2010/2011 Exhibits 1027-1042., number 348 in Asterisk, 375--403 (2012).

\bibitem{HM22}
D.~Huybrechts and M.~Mauri, \emph{Lagrangian fibrations}, Milan J. Math. \textbf{90} (2022), no.~2,  459--483.

\bibitem{HM23}
---{}---{}---, \emph{On type {II} degenerations of hyperk\"{a}hler manifolds}, Math. Res. Lett. \textbf{30} (2023), no.~1,  125--141.

\bibitem{yoonjoo-base}
Y.-J. Kim, \emph{Note on {OpenAI}'s proof of the base of {Lagrangian} fibrations} (2026).

\bibitem{KLSV2018}
J.~Koll\'{a}r, R.~Laza, G.~Sacc\`a, and C.~Voisin, \emph{Remarks on degenerations of hyper-{K}\"{a}hler manifolds}, Ann. Inst. Fourier (Grenoble) \textbf{68} (2018), no.~7,  2837--2882.

\bibitem{li2026}
Y.~Li, \emph{Valuative independence and metric SYZ conjecture}, arXiv:2605.00516 (2026).

\bibitem{Mat2001}
D.~Matsushita, \emph{Addendum: ``{O}n fibre space structures of a projective irreducible symplectic manifold'' [{T}opology {\bf 38} (1999), no. 1, 79--83; {MR}1644091 (99f:14054)]}, Topology \textbf{40} (2001), no.~2,  431--432.

\bibitem{Matsushita}
---{}---{}---, \emph{On deformations of {L}agrangian fibrations}, in K3 surfaces and their moduli, Vol. 315 of \emph{Progr. Math.}, 237--243 (2016).

\bibitem{OO2021}
Y.~Odaka and Y.~Oshima, Collapsing {K}3 surfaces, tropical geometry and moduli compactifications of {S}atake, {M}organ-{S}halen type, Vol.~40 of \emph{MSJ Memoirs}, Mathematical Society of Japan, Tokyo (2021).

\bibitem{oai-base}
{OpenAI}, \emph{Projective-space bases of {Lagrangian} fibrations} (2026). OpenAI Math Release preprint.

\bibitem{oai-syz}
---{}---{}---, \emph{The strong hyperk{\"a}hler {SYZ} conjecture} (2026). OpenAI Math Release preprint.

\bibitem{Sawon2003}
J.~Sawon, \emph{Abelian fibred holomorphic symplectic manifolds}, Turkish J. Math. \textbf{27} (2003), no.~1,  197--230.

\bibitem{Schmid}
W.~Schmid, \emph{Variation of {H}odge structure: the singularities of the period mapping}, Invent. Math. \textbf{22} (1973) 211--319.

\bibitem{Sold2019}
A.~Soldatenkov, \emph{Limit mixed {H}odge structures of hyperk\"{a}hler manifolds}, Mosc. Math. J. \textbf{20} (2020), no.~2,  423--436.

\bibitem{SV2025}
A.~Soldatenkov and M.~Verbitsky, \emph{The abundance and {SYZ} conjectures in families of hyperk\"{a}hler manifolds}, \'{E}pijournal G\'{e}om. Alg\'{e}brique \textbf{9} (2025) Art. 26, 17.

\bibitem{TZ2020}
V.~Tosatti and Y.~Zhang, \emph{Collapsing {H}yperk\"{a}hler manifolds}, Ann. Sci. \'{E}c. Norm. Sup\'{e}r. (4) \textbf{53} (2020), no.~3,  751--786.

\bibitem{Ver96}
M.~Verbitsky, \emph{Cohomology of compact hyper-{K}\"{a}hler manifolds and its applications}, Geom. Funct. Anal. \textbf{6} (1996), no.~4,  601--611.

\bibitem{Verb2010}
---{}---{}---, \emph{Hyper{K}\"{a}hler {SYZ} conjecture and semipositive line bundles}, Geom. Funct. Anal. \textbf{19} (2010), no.~5,  1481--1493.

\bibitem{Verb13}
---{}---{}---, \emph{Mapping class group and a global {T}orelli theorem for hyperk\"{a}hler manifolds}, Duke Math. J. \textbf{162} (2013), no.~15,  2929--2986. Appendix A by Eyal Markman.

\bibitem{Verb20}
---{}---{}---, \emph{Errata for ``{M}apping class group and a global {T}orelli theorem for hyperk\"{a}hler manifolds'' by {M}isha {V}erbitsky}, Duke Math. J. \textbf{169} (2020), no.~5,  1037--1038.

\bibitem{zhu2023}
Y.~Zhu, \emph{The general definition of Kahler angle}, arXiv:2303.17122 (2023).

\end{thebibliography}
\bibliographystyle{mrl}

\end{document}